\documentclass[a4paper,12pt]{article}
\usepackage{amsmath,amssymb,amsthm,amsfonts}
\usepackage[utf8]{inputenc}
\usepackage[T1]{fontenc}
\usepackage{anysize}
\marginsize{2cm}{1.5cm}{2cm}{1.5cm}

\title{Generalized Beta Function Defined by Wright Function}
\author{ENES ATA \\ 
	\small Ahi Evran University, Department of Mathematics, Kırşehir, Turkey. \\ \small enesata.tr@gmail.com}
\date{}

\newtheorem{tnm}{Definition}[section]
\newtheorem{teo}[tnm]{Theorem}
\newtheorem{snc}[tnm]{Corollary}
\begin{document}
\maketitle

\begin{abstract}
The main object of this paper is to present generalizations of gamma, beta and hypergeometric functions. Some recurrence relations, transformation formulas, operation formulas and integral representations are obtained for these new generalizations.
\\

\end{abstract}

\ \section{Introduction}
In recent years, some extensions of the well known special functions have been considered by several authors \cite{1,2,3,4,5,6,8,9}. In 1994, Chaudhry and Zubair \cite{1} have indroduced the following extension of gamma function

\begin{align} \label{1}
\Gamma_p (x)= \int_{0}^{\infty} t^{x-1} \ \exp{\left(-t-\frac{p}{t}\right)}dt,\quad\left( Re(p)>0\right) .
\end{align}
In 1997, Chaudhry et al. \cite{4} presented the following extension of Euler's beta function
\begin{align} \label{2}
B_p(x,y)=\int_{0}^{1} t^{x-1}(1-t)^{y-1} \exp\left(-\frac{p}{t(1-t)}\right) dt
\end{align}
$$\left(  Re(p)>0,Re(x)>0,Re(y)>0\right)$$
and they proved that this extension has connections with the Macdonald, error and Whittakers function. It clearly seems that $\Gamma_{0}(x)=\Gamma(x)$ and $B_{0}(x,y)=B(x,y)$, where $\Gamma(x)$ and $B(x,y)$ are the classical gamma and beta functions \cite{10}. Afterwards, Chaudhry et al. \cite{6} used $B_{p}(x,y)$ to extend the Gauss hypergeometric function and the confluent hypergeometric function as follows:
\begin{align} \label{3}
{}_{}\textrm{F}_{p}(a,b;c;z)= \sum_{n=0}^{\infty} (a)_n \frac{B_{p}(b+n,c-b)}{B(b,c-b)}  \frac{z^n}{n!}   
\end{align} 
$$\left( p\geq0;Re(c)>Re(b)>0\right),$$
\begin{align} \label{4}
{}_{}{\Phi}_{p}(b;c;z)= \sum_{n=0}^{\infty}\frac{B_{p}(b+n,c-b)}{B(b,c-b)}  \frac{z^n}{n!} 
\end{align}$$\left( p\geq0;Re(c)>Re(b)>0\right),$$
where $(\lambda)_{v}$ denotes the Pochhammer symbol defined by
$$(\lambda)_{0} \equiv 1~~\textnormal{and}~~ (\lambda)_{v}=\frac{\Gamma(\lambda+v)}{\Gamma(\lambda)}$$
and gave the Euler type integral representations
\begin{align*} 
{}_{}\textrm{F}_{p}(a,b;c;z)= \frac{1}{B(b, c-b)} \int_{0}^{1} t^{b-1} (1-t)^{c-b-1}(1-zt)^{-a}\ \exp\left(-\frac{p}{t(1-t)}\right) dt
\end{align*}
$$\left( p>0;p=0~\textnormal{and}~|\arg(1-z)|<\pi<p;Re(c)>Re(b)>0\right),$$
\begin{align*}
{}_{}{\Phi}_{p}(b;c;z)= \frac{1}{B(b, c-b)} \int_{0}^{1} t^{b-1} (1-t)^{c-b-1}\ \exp\left(zt-\frac{p}{t(1-t)}\right) dt
\end{align*}
$$\left( p>0;p=0~\textnormal{and}~Re(c)>Re(b)>0\right).$$
They called these functions  extended Gauss hypergeometric function (EGHF) and extended confluent hypergeometric function (ECHF), respectively. They have discussed the differentiation properties and Mellin transform of ${}_{}\textrm{F}_{p}(a,b;c;z)$ and ${}_{}{\Phi}_{p}(b;c;z)$ and obtained transformation  formulas, recurrence relations, summation and asymptotic formulas for these functions.

In this paper, we use Wright function to define new generalizations of gamma and beta functions, which defined in \cite{7} as
\begin{align*} 
{}_{0}\Psi_{1}(\alpha,\beta;z)=\sum_{n=0}^{\infty} \frac{1}{\Gamma(\alpha n + \beta)}\frac{z^n}{n!},
\end{align*} 
where $\alpha$, $\beta$ $\in\mathbb{C}$ and $Re(\alpha)>-1$.\\

We start our investigation introducing the following generalizations of gamma and beta functions as
\begin{align}\label{gm1}
{}^{\Psi}\Gamma_{p}^{(\alpha,\beta)}(x)=\int_{0}^{\infty}t^{x-1}{}_{0}\Psi_{1}\left( \alpha,\beta;-t-\frac{p}{t}\right) dt
\end{align}
$$\left( Re(x)>0,Re(\alpha)>-1,Re(p)>0\right),$$
and
\begin{align}\label{bt1}
{}^{\Psi}B_{p}^{(\alpha,\beta)}(x,y)=\int_{0}^{1}t^{x-1}(1-t)^{y-1}{}_{0}\Psi_{1}\left( \alpha,\beta;-\frac{p}{t(1-t)}\right) dt
\end{align}
$$\left( Re(x)>0,Re(y)>0,Re(\alpha)>-1,Re(p)>0\right).$$
We call the new generalizations of gamma and beta functions as $\Psi$-gamma and $\Psi$-beta functions, respectively. It is obvious that,

 $${}^{\Psi}\Gamma_{p}^{(0,2)}(x)=\Gamma_{p}(x),$$ $${}^{\Psi}\Gamma_{0}^{(0,2)}(x)=\Gamma(x),$$
and $${}^{\Psi}B_{p}^{(0,2)}(x,y)=B_{p}(x,y),$$ $${}^{\Psi}B_{0}^{(0,2)}(x,y)=B(x,y),$$ where $\Gamma_p(x), B_p(x,y)$ are generalized functions defined as \eqref{1} and \eqref{2}, respectively.

In  section 2, different integral representations and properties of $\Psi$-beta function are obtained. Additionally, relations of $\Psi$-gamma and $\Psi$-beta functions are discussed. In Section 3, we define $\Psi$-Gauss and $\Psi$-confluent hypergeometric functions, using ${}^{\Psi}B_{p}^{(\alpha,\beta)}(x,y)$ obtain the integral representations  of these $\Psi$-Gauss and $\Psi$-confluent hypergeometric functions. Furthermore we discussed the differentiation properties, Mellin transforms, transformation formulas, recurrence relations, summation formulas for these new hypergeometric functions. 

\ \section{Properties of $\Psi$-gamma and $\Psi$-beta functions}
The next theorem gives the Mellin transform representation of the function ${}^{\Psi}B_{p}^{(\alpha,\beta)}(x,y)$ in terms of the ordinary beta function and ${}^{\Psi}\Gamma^{(\alpha,\beta)}(s)$.
\begin{teo}
Mellin transform representation of the $\Psi$-beta function is given by
\begin{align*}
\mathcal{M}\left\lbrace  {}^{\Psi}B_{p}^{(\alpha,\beta)}(x,y)\right\rbrace =B(x+s,y+s){}^{\Psi}\Gamma^{(\alpha,\beta)}(s)
\end{align*}
$$\left( Re(s)>0,Re(x+s)>0,Re(y+s)>0,Re(p)>0,Re(\alpha)>-1\right).$$
\end{teo}

\begin{proof}
Multiplying \eqref{bt1} by $p^{s-1}$ and integrating with respect to $p$ from $p=0$ to $p=\infty$, we get
\begin{align}\label{f1}
\mathcal{M}\left\lbrace  {}^{\Psi}B_{p}^{(\alpha,\beta)}(x,y)\right\rbrace  =\int_{0}^{\infty}p^{s-1}\int_{0}^{1}t^{x-1}(1-t)^{y-1}{}_{0}\Psi_{1}\left( \alpha,\beta;-\frac{p}{t(1-t)}\right) dtdp.
\end{align}
From the uniform convergence of the integral, the order of integration in \eqref{f1} can be interchanged. Therefore, we have
\begin{align}\label{f2}
\mathcal{M}\left\lbrace  {}^{\Psi}B_{p}^{(\alpha,\beta)}(x,y)\right\rbrace =\int_{0}^{1}t^{x-1}(1-t)^{y-1}\int_{0}^{\infty}p^{s-1}{}_{0}\Psi_{1}\left( \alpha,\beta;-\frac{p}{t(1-t)}\right)dpdt.
\end{align}
Now using the one-to-one transformation (except possibly at the boundaries and maps of the region onto itself) $v=\frac{p}{t(1-t)}$ in \eqref{f2}, we get,
\begin{align*}
\mathcal{M}\left\lbrace  {}^{\Psi}B_{p}^{(\alpha,\beta)}(x,y)\right\rbrace =\int_{0}^{1}t^{x+s-1}(1-t)^{y+s-1}dt\int_{0}^{\infty}v^{s-1}{}_{0}\Psi_{1}\left( \alpha,\beta;-v\right)dv.
\end{align*}
Therefore, we have
\begin{align*}
\mathcal{M}\left\lbrace  {}^{\Psi}B_{p}^{(\alpha,\beta)}(x,y)\right\rbrace =B(x+s,y+s){}^{\Psi}\Gamma^{(\alpha,\beta)}(s),
\end{align*} 
which completes the proof.
\end{proof}

\begin{snc}
By the Mellin inversion formula, we have the following complex integral representation for ${}^{\Psi}B_{p}^{(\alpha,\beta)}(x,y):$
\begin{align*}
{}^{\Psi}B_{p}^{(\alpha,\beta)}(x,y)=\frac{1}{2\pi i}\int_{-i\infty}^{+i\infty}B(x+s,y+s){}^{\Psi}\Gamma^{(\alpha,\beta)}(s)p^{-s}ds.
\end{align*}
\end{snc}

\begin{teo}
For the $\Psi$-beta function, we have the following integral representations:
\begin{align*}
{}^{\Psi}B_{p}^{(\alpha,\beta)}(x,y)&=2\int_{0}^{\frac{\pi}{2}} \cos^{2x-1}(\theta) \sin^{2y-1}(\theta){}_{0}\Psi_{1}\left( \alpha,\beta;-p \sec^{2}(\theta) \csc^{2}(\theta)\right)d\theta,
\\
{}^{\Psi}B_{p}^{(\alpha,\beta)}(x,y)&=\int_{0}^{\infty}\frac{u^{x-1}}{(1+u)^{x+y}}{}_{0}\Psi_{1}\left( \alpha,\beta;-2p-p\left( u+\frac{1}{u}\right) \right)du.
\end{align*}
\end{teo} 
\begin{proof}
Letting $t= \cos^{2}(\theta)$ in \eqref{bt1}, we get
\begin{align*}
{}^{\Psi}B_{p}^{(\alpha,\beta)}(x,y)&=\int_{0}^{1}t^{x-1}(1-t)^{y-1}{}_{0}\Psi_{1}\left( \alpha,\beta;-\frac{p}{t(1-t)}\right) dt\\
&=2\int_{0}^{\frac{\pi}{2}} \cos^{2x-1}(\theta) \sin^{2y-1}(\theta){}_{0}\Psi_{1}\left( \alpha,\beta;-p \sec^{2}(\theta) \csc^{2}(\theta)\right)d\theta.
\end{align*}
On the other hand, letting $t=\frac{u}{1+u}$ in \eqref{bt1}, we get 
\begin{align*}
{}^{\Psi}B_{p}^{(\alpha,\beta)}(x,y)&=\int_{0}^{1}t^{x-1}(1-t)^{y-1}{}_{0}\Psi_{1}\left( \alpha,\beta;-\frac{p}{t(1-t)}\right) dt\\
&=\int_{0}^{\infty} \frac{u^{x-1}}{(1+u)^{x+y}}{}_{0}\Psi_{1}\left( \alpha,\beta;-2p-p\left( u+\frac{1}{u}\right) \right)du,
\end{align*}
which completes the proof.	
\end{proof}

\begin{teo}
For the $\Psi$-beta function, we have the following functional relation:
\begin{align*}
{}^{\Psi}B_{p}^{(\alpha,\beta)}(x,y+1)+{}^{\Psi}B_{p}^{(\alpha,\beta)}(x+1,y)={}^{\Psi}B_{p}^{(\alpha,\beta)}(x,y).
\end{align*}
\end{teo} 
\begin{proof}
Direct calculations yield	
\begin{align*}
{}^{\Psi}B_{p}^{(\alpha,\beta)}&(x,y+1)+{}^{\Psi}B_{p}^{(\alpha,\beta)}(x+1,y)\\
&=\int_{0}^{1}t^{x-1}(1-t)^{y}{}_{0}\Psi_{1}\left( \alpha,\beta;-\frac{p}{t(1-t)}\right) dt+\int_{0}^{1}t^{x}(1-t)^{y-1}{}_{0}\Psi_{1}\left( \alpha,\beta;-\frac{p}{t(1-t)}\right) dt\\
&=\int_{0}^{1}\left[ t^{x-1}(1-t)^{y}+t^{x}(1-t)^{y-1}\right]{}_{0}\Psi_{1}\left( \alpha,\beta;-\frac{p}{t(1-t)}\right) dt\\
&=\int_{0}^{1}t^{x-1}(1-t)^{y-1}{}_{0}\Psi_{1}\left( \alpha,\beta;-\frac{p}{t(1-t)}\right) dt\\
&={}^{\Psi}B_{p}^{(\alpha,\beta)}(x,y),
\end{align*}
hence the result.	
\end{proof}

\begin{teo}
For the product of two $\Psi$-gamma function, we have the following integral representation:
\begin{align*}
{}^{\Psi}\Gamma_{p}^{(\alpha,\beta)}(x){}^{\Psi}\Gamma_{p}^{(\alpha,\beta)}(y)&=4\int_{0}^{\frac{\pi}{2}}\int_{0}^{\infty}r^{2(x+y)-1} \cos^{2x-1}(\theta) \sin^{2y-1}(\theta)\\
&~~~\times{}_{0}\Psi_{1}\left( \alpha,\beta;-r^{2} \cos^{2}(\theta)-\frac{p}{r^{2} \cos^{2}(\theta)}\right)\\
&~~~\times{}_{0}\Psi_{1}\left( \alpha,\beta;-r^{2} \sin^{2}(\theta)-\frac{p}{r^{2} \sin^{2}(\theta)}\right)drd\theta.
\end{align*}
\end{teo} 
\begin{proof}
Substituting $t=\eta^{2}$ in \eqref{gm1}, we get
\begin{align*}
{}^{\Psi}\Gamma_{p}^{(\alpha,\beta)}(x)=2\int_{0}^{\infty}\eta^{2x-1}{}_{0}\Psi_{1}\left( \alpha,\beta;-\eta^{2}-\frac{p}{\eta^{2}}\right)d \eta.
\end{align*}
Therefore,
\begin{align*}
{}^{\Psi}\Gamma_{p}^{(\alpha,\beta)}(x){}^{\Psi}\Gamma_{p}^{(\alpha,\beta)}(y)&=4\int_{0}^{\infty}\int_{0}^{\infty}\eta^{2x-1}\xi^{2y-1}{}_{0}\Psi_{1}\left( \alpha,\beta;-\eta^{2}-\frac{p}{\eta^{2}}\right)
\\
&~~~\times{}_{0}\Psi_{1}\left( \alpha,\beta;-\xi^{2}-\frac{p}{\xi^{2}}\right)d\eta d\xi.
\end{align*}
Letting $\eta=r \cos(\theta)$ and $\xi=r \sin(\theta)$ in the above equality,
\begin{align*}
{}^{\Psi}\Gamma_{p}^{(\alpha,\beta)}(x){}^{\Psi}\Gamma_{p}^{(\alpha,\beta)}(y)&=4\int_{0}^{\frac{\pi}{2}}\int_{0}^{\infty}r^{2(x+y)-1} \cos^{2x-1}(\theta) \sin^{2y-1}(\theta)\\
&~~~\times{}_{0}\Psi_{1}\left( \alpha,\beta;-r^{2} \cos^{2}(\theta)-\frac{p}{r^{2} \cos^{2}(\theta)}\right)\\
&~~~\times{}_{0}\Psi_{1}\left( \alpha,\beta;-r^{2} \sin^{2}(\theta)-\frac{p}{r^{2} \sin^{2}(\theta)}\right)drd\theta,
\end{align*}
completes the proof.
\end{proof}

\begin{teo}
For the $\Psi$-beta function, we have the following summation relation:
\begin{align*}
{}^{\Psi}B_{p}^{(\alpha,\beta)}(x,1-y)=\sum_{n=0}^{\infty}\frac{(y)_{n}}{n!}{}^{\Psi}B_{p}^{(\alpha,\beta)}(x+n,1),\quad\left( Re(p)>0\right) .
\end{align*}
\end{teo} 
\begin{proof}
From the definition of the $\Psi$-beta function, we get
\begin{align*}
{}^{\Psi}B_{p}^{(\alpha,\beta)}(x,1-y)=\int_{0}^{1}t^{x-1}(1-t)^{-y}{}_{0}\Psi_{1}\left( \alpha,\beta;-\frac{p}{t(1-t)}\right) dt.
\end{align*}
Using the following binomial series expansion
\begin{align*}
(1-t)^{-y}=\sum_{n=0}^{\infty}(y)_{n}\frac{t^{n}}{n!},\quad|t|<1,
\end{align*}
we obtain
\begin{align*}
{}^{\Psi}B_{p}^{(\alpha,\beta)}(x,1-y)=\int_{0}^{1}\sum_{n=0}^{\infty}\frac{(y)_{n}}{n!}t^{x+n-1}{}_{0}\Psi_{1}\left( \alpha,\beta;-\frac{p}{t(1-t)}\right) dt.
\end{align*}
Therefore, interchanging the order of integration and summation and then using \eqref{bt1}, we obtain
\begin{align*}
{}^{\Psi}B_{p}^{(\alpha,\beta)}(x,1-y)&=\sum_{n=0}^{\infty}\frac{(y)_{n}}{n!}\int_{0}^{1}t^{x+n-1}{}_{0}\Psi_{1}\left( \alpha,\beta;-\frac{p}{t(1-t)}\right) dt\\
&=\sum_{n=0}^{\infty}\frac{(y)_{n}}{n!} {}^{\Psi}B_{p}^{(\alpha,\beta)}(x+n,1),
\end{align*}
which completes the proof.
\end{proof}

\section{$\Psi$-generalization of Gauss and confluent hypergeometric functions}
In this  section, we use the $\Psi$-beta function \eqref{bt1} to define $\Psi$-generalization of Gauss and confluent hypergeometric functions as
\begin{align*}
{}^{\Psi}F_{p}^{(\alpha,\beta)}(a,b;c;z)=\sum_{n=0}^{\infty}(a)_{n}\frac{{}^{\Psi}B_{p}^{(\alpha,\beta)}(b+n,c-b)}{B(b,c-b)}\frac{z^{n}}{n!}
\end{align*} and
\begin{align*}
{}^{\Psi}\Phi_{p}^{(\alpha,\beta)}(b;c;z)=\sum_{n=0}^{\infty}\frac{{}^{\Psi}B_{p}^{(\alpha,\beta)}(b+n,c-b)}{B(b,c-b)}\frac{z^{n}}{n!},
\end{align*} respectively.

We call ${}^{\Psi}F_{p}^{(\alpha,\beta)}(a,b;c;z)$ as $\Psi$-Gauss hypergeometric function and ${}^{\Psi}\Phi_{p}^{(\alpha,\beta)}(b;c;z)$ as $\Psi$-confluent hypergeometric function.

Observe that 
$${}^{\Psi}F_{p}^{(0,2)}(a,b;c;z)=F_{p}(a,b;c;z),$$
$${}^{\Psi}F_{0}^{(0,2)}(a,b;c;z)={}_{2}F_{1}(a,b;c;z),$$
and
$${}^{\Psi}\Phi_{p}^{(0,2)}(b;c;z)=\Phi_{p}(b;c;z),$$
$${}^{\Psi}\Phi_{0}^{(0,2)}(b;c;z)=\Phi(b;c;z),$$
where $F_p(a,b;c;z), \Phi_p(b;c;z)$ are generalized functions defined as \eqref{3} and \eqref{4}, respectively.

\subsection{Integral representations}
The $\Psi$-Gauss hypergeometric function can be provided with an integral representation by using the definition of the $\Psi$-beta function \eqref{bt1}.

\begin{teo} Let $Re(p)>0;p=0~\textnormal{and}~|\arg(1-z)|<\pi;Re(c)>Re(b)>0,$ then
\begin{align}\label{j1}
{}^{\Psi}F_{p}^{(\alpha,\beta)}(a,b;c;z)&=\frac{1}{B(b,c-b)}\int_{0}^{1}t^{b-1}(1-t)^{c-b-1}{}_{0}\Psi_{1}\left( \alpha,\beta;-\frac{p}{t(1-t)}\right)(1-zt)^{-a}dt,
\\
{}^{\Psi}F_{p}^{(\alpha,\beta)}(a,b;c;z)&=\frac{1}{B(b,c-b)}\int_{0}^{\infty}u^{b-1}(1+u)^{a-c}\left[ 1+u(1-z)\right]^{-a}\nonumber
\\
&~~~\times{}_{0}\Psi_{1}\left( \alpha,\beta;-2p-p\left( u+\frac{1}{u}\right) \right)du,\nonumber
\\
{}^{\Psi}F_{p}^{(\alpha,\beta)}(a,b;c;z)&=\frac{2}{B(b,c-b)}\int_{0}^{\frac{\pi}{2}} \sin^{2b-1}(\theta) \cos^{2c-2b-1}(\theta)\left( 1-z \sin^{2}(\theta)\right) ^{-a}\nonumber
\\
&~~~\times{}_{0}\Psi_{1}\left( \alpha,\beta;-\frac{p}{ \sin^{2}(\theta) \cos^{2}(\theta)}\right)d\theta.\nonumber
\end{align}
\end{teo}

\begin{proof}
Direct calculations yield	
\begin{align*}
{}^{\Psi}F_{p}^{(\alpha,\beta)}(a,b;c;z)&=\sum_{n=0}^{\infty}(a)_{n}\frac{{}^{\Psi}B_{p}^{(\alpha,\beta)}(b+n,c-b)}{B(b,c-b)}\frac{z^{n}}{n!}\\
&=\frac{1}{B(b,c-b)}\sum_{n=0}^{\infty}(a)_{n}\int_{0}^{1}t^{b+n-1}(1-t)^{c-b-1}{}_{0}\Psi_{1}\left( \alpha,\beta;-\frac{p}{t(1-t)}\right) \frac{z^{n}}{n!}dt\\
&=\frac{1}{B(b,c-b)}\int_{0}^{1}t^{b-1}(1-t)^{c-b-1}{}_{0}\Psi_{1}\left( \alpha,\beta;-\frac{p}{t(1-t)}\right) \sum_{n=0}^{\infty}(a)_{n}\frac{(zt)^{n}}{n!}dt\\
&=\frac{1}{B(b,c-b)}\int_{0}^{1}t^{b-1}(1-t)^{c-b-1}{}_{0}\Psi_{1}\left( \alpha,\beta;-\frac{p}{t(1-t)}\right) (1-zt)^{-a}dt.
\end{align*}
Setting $u=\frac{t}{1-t}$ in \eqref{j1}, we get
\begin{align*}
{}^{\Psi}F_{p}^{(\alpha,\beta)}(a,b;c;z)=\frac{1}{B(b,c-b)}\int_{0}^{\infty}u^{b-1}(1+u)^{a-c}\left[ 1+u(1-z)\right]^{-a}{}_{0}\Psi_{1}\left( \alpha,\beta;-2p-p\left( u+\frac{1}{u}\right) \right)du.
\end{align*}
On the other hand, substituting $t= \sin^{2}(\theta)$ in \eqref{j1}, we have
\begin{align*}
{}^{\Psi}F_{p}^{(\alpha,\beta)}(a,b;c;z)&=\frac{2}{B(b,c-b)}\int_{0}^{\frac{\pi}{2}} \sin^{2b-1}(\theta) \cos^{2c-2b-1}(\theta)\left( 1-z \sin^{2}(\theta)\right) ^{-a}\\
&~~~\times{}_{0}\Psi_{1}\left( \alpha,\beta;-\frac{p}{ \sin^{2}(\theta) \cos^{2}(\theta)}\right)d\theta.\qedhere
\end{align*} 
\end{proof}

A smilar procedure yields an integral representation of the $\Psi$-confluent hypergeometric function by using the definition of the $\Psi$-beta function.

\begin{teo}
For the $\Psi$-confluent hypergeometric function, we have the following integral representations:	
\begin{align}\label{c1}
{}^{\Psi}\Phi_{p}^{(\alpha,\beta)}(b;c;z)&=\frac{1}{B(b,c-b)}\int_{0}^{1}t^{b-1}(1-t)^{c-b-1}\exp(zt){}_{0}\Psi_{1}\left( \alpha,\beta;-\frac{p}{t(1-t)}\right)dt,
\\
{}^{\Psi}\Phi_{p}^{(\alpha,\beta)}(b;c;z)&=\frac{1}{B(b,c-b)}\int_{0}^{1}u^{c-b-1}(1-u)^{b-1}\exp(z(1-u)){}_{0}\Psi_{1}\left( \alpha,\beta;-\frac{p}{u(1-u)}\right) du,\nonumber
\end{align}
$$\left(p\geq0;~\textnormal{and}~Re(c)>Re(b)>0\right).$$
\end{teo}

\subsection{Differentiation formulas}
In this  section, by using the formulas,
$$B(b,c-b)=\frac{c}{b}B(b+1,c-b)$$and $$(a)_{n+1}=a(a+1)_{n}$$
we obtain new formulas including derivatives of $\Psi$-Gauss hypergeometric function and $\Psi$-confluent hypergeometric function with respect to the variable $z$.

\begin{teo}
For $\Psi$-Gauss hypergeometric function, we have the following differentiation formula:
	\begin{align*}
	\frac{d^{n}}{dz^{n}}\left\lbrace {}^{\Psi}F_{p}^{(\alpha,\beta)}(a,b;c;z)\right\rbrace =\frac{(a)_{n}(b)_{n}}{(c)_{n}}\left[ {}^{\Psi}F_{p}^{(\alpha,\beta)}(a+n,b+n;c+n;z)\right].
	\end{align*}
\end{teo}
\begin{proof}
Taking the derivative of ${}^{\Psi}F_{p}^{(\alpha,\beta)}(a,b;c;z)$ with respect to $z$, we obtain
\begin{align*}
\frac{d}{dz}\left\lbrace {}^{\Psi}F_{p}^{(\alpha,\beta)}(a,b;c;z)\right\rbrace &=\frac{d}{dz}\left\lbrace \sum_{n=0}^{\infty}(a)_{n}\frac{{}^{\Psi}B_{p}^{(\alpha,\beta)}(b+n,c-b)}{B(b,c-b)}\frac{z^{n}}{n!}\right\rbrace \\
&=\sum_{n=1}^{\infty}(a)_{n}\frac{{}^{\Psi}B_{p}^{(\alpha,\beta)}(b+n,c-b)}{B(b,c-b)}\frac{z^{n-1}}{(n-1)!}. 
\end{align*} Replacing $n\to n+1$, we get
\begin{align*}
\frac{d}{dz}\left\lbrace {}^{\Psi}F_{p}^{(\alpha,\beta)}(a,b;c;z)\right\rbrace &=\frac{(a)(b)}{(c)}\sum_{n=0}^{\infty}(a+1)_{n}\frac{{}^{\Psi}B_{p}^{(\alpha,\beta)}(b+n+1,c-b)}{B(b+1,c-b)}\frac{z^{n}}{n!}\\
&=\frac{(a)(b)}{(c)}\left[ {}^{\Psi}F_{p}^{(\alpha,\beta)}(a+1,b+1;c+1;z)\right].
\end{align*} Recursive application of this procedure gives us the general form:
\begin{align*}
\frac{d^{n}}{dz^{n}}\left\lbrace {}^{\Psi}F_{p}^{(\alpha,\beta)}(a,b;c;z)\right\rbrace =\frac{(a)_{n}(b)_{n}}{(c)_{n}}\left[ {}^{\Psi}F_{p}^{(\alpha,\beta)}(a+n,b+n;c+n;z)\right],
\end{align*}
which completes the proof.
\end{proof}

\begin{teo}
For $\Psi$-confluent hypergeometric function, we have the following differentiation formula:
\begin{align*}
\frac{d^{n}}{dz^{n}}\left\lbrace {}^{\Psi}\Phi_{p}^{(\alpha,\beta)}(b;c;z)\right\rbrace =\frac{(b)_{n}}{(c)_{n}}\left[ {}^{\Psi}\Phi_{p}^{(\alpha,\beta)}(b+n;c+n;z)\right].
\end{align*}
\end{teo} 
\begin{proof}
Taking the derivative of ${}^{\Psi}\Phi_{p}^{(\alpha,\beta)}(b;c;z)$ with respect to $z$, we obtain
\begin{align*}
\frac{d}{dz}\left\lbrace {}^{\Psi}\Phi_{p}^{(\alpha,\beta)}(b;c;z)\right\rbrace &=\frac{d}{dz}\left\lbrace \sum_{n=0}^{\infty}\frac{{}^{\Psi}B_{p}^{(\alpha,\beta)}(b+n,c-b)}{B(b,c-b)}\frac{z^{n}}{n!}\right\rbrace \\
&=\sum_{n=1}^{\infty}\frac{{}^{\Psi}B_{p}^{(\alpha,\beta)}(b+n,c-b)}{B(b,c-b)}\frac{z^{n-1}}{(n-1)!}. 
\end{align*} Replacing $n\to n+1$, we get
\begin{align*}
\frac{d}{dz}\left\lbrace {}^{\Psi}\Phi_{p}^{(\alpha,\beta)}(b;c;z)\right\rbrace &=\frac{(b)}{(c)}\sum_{n=0}^{\infty}\frac{{}^{\Psi}B_{p}^{(\alpha,\beta)}(b+n+1,c-b)}{B(b+1,c-b)}\frac{z^{n}}{n!}\\
&=\frac{(b)}{(c)}\left[ {}^{\Psi}\Phi_{p}^{(\alpha,\beta)}(b+1;c+1;z)\right].
\end{align*} Recursive application of this procedure gives us the general form:
\begin{align*}
\frac{d^{n}}{dz^{n}}\left\lbrace {}^{\Psi}\Phi_{p}^{(\alpha,\beta)}(b;c;z)\right\rbrace =\frac{(b)_{n}}{(c)_{n}}\left[ {}^{\Psi}\Phi_{p}^{(\alpha,\beta)}(b+n;c+n;z)\right],
\end{align*}
which completes the proof.
\end{proof}

\subsection{Mellin transform representations}
In this  section, we obtain the Mellin transform representations of the $\Psi$-Gauss and $\Psi$-confluent hypergeometric functions.

\begin{teo}
For the $\Psi$-Gauss hypergeometric function, we have the following Mellin transform representation:
\begin{align*}
\mathcal{M}\left\lbrace  {}^{\Psi}F_{p}^{(\alpha,\beta)}(a,b;c;z):s\right\rbrace =\frac{{}^{\Psi}\Gamma^{(\alpha,\beta)}(s)B(b+s,c+s-b)}{B(b,c-b)}{}_{2}F_{1}(a,b+s;c+2s;z).
\end{align*}
\end{teo}
\begin{proof}
To obtain the Mellin transform, we multiply both sides of \eqref{j1} by $p^{s-1}$ and integrate with respect to $p$ over the interval $[0,\infty).$ Thus, we get
\begin{align*}
	\mathcal{M}\big\lbrace {}^{\Psi}F_{p}^{(\alpha,\beta)}&(a,b;c;z):s\big\rbrace
	\\
	&=\int_{0}^{\infty}p^{s-1} ~{}^{\Psi}F_{p}^{(\alpha,\beta)}(a,b;c;z)dp \\
	&=\int_{0}^{\infty}p^{s-1}\sum_{n=0}^{\infty}\frac{{}^{\Psi}B_{p}^{(\alpha,\beta)}(b+n,c-b)}{B(b,c-b)}(a)_{n}\frac{z^{n}}{n!}dp\\
	&=\frac{1}{B(b,c-b)}\int_{0}^{1}t^{b-1}(1-t)^{c-b-1}(1-zt)^{-a}\int_{0}^{\infty}p^{s-1}{}_{0}\Psi_{1}\left( \alpha,\beta;-\frac{p}{t(1-t)}\right)dpdt.\\
	\end{align*}
Since substituting $u=\frac{p}{t(1-t)}$	in the above equation,
\begin{align*}
\int_{0}^{\infty}p^{s-1}{}_{0}\Psi_{1}\left( \alpha,\beta;-\frac{p}{t(1-t)}\right)dp=t^{s}(1-t)^{s}{}^{\Psi}\Gamma^{(\alpha,\beta)}(s).
\end{align*}
Thus, we get 
\begin{align*}
\mathcal{M}\left\lbrace  {}^{\Psi}F_{p}^{(\alpha,\beta)}(a,b;c;z):s\right\rbrace =\frac{{}^{\Psi}\Gamma^{(\alpha,\beta)}(s)B(b+s,c+s-b)}{B(b,c-b)}{}_{2}F_{1}(a,b+s;c+2s;z).~~~~~~~~~~~~~~~~~~~~~~~~\qedhere
\end{align*}
\end{proof}

\begin{snc}
By the Mellin inversion formula, we have the following complex integral representation for ${}^{\Psi}F_{p}^{(\alpha,\beta)}(a,b;c;z)$:
\begin{align*}
{}^{\Psi}F_{p}^{(\alpha,\beta)}(a,b;c;z)=\frac{1}{2\pi i}\int_{-i\infty}^{+i\infty}\frac{{}^{\Psi}\Gamma^{(\alpha,\beta)}(s)B(b+s,c+s-b)}{B(b,c-b)}{}_{2}F_{1}(a,b+s;c+2s;z)p^{-s}ds.
\end{align*}
\end{snc}

\begin{teo}
For the $\Psi$-confluent hypergeometric function, we have the following Mellin transform representation:
	\begin{align*}
	\mathcal{M}\left\lbrace  {}^{\Psi}\Phi_{p}^{(\alpha,\beta)}(b;c;z):s\right\rbrace =\frac{{}^{\Psi}\Gamma^{(\alpha,\beta)}(s)B(b+s,c+s-b)}{B(b,c-b)}\Phi(b+s;c+2s;z).
	\end{align*}
\end{teo}
\begin{proof}
To obtain the Mellin transform, we multiply both sides of \eqref{c1} by $p^{s-1}$ and integrate with respect to $p$ over the interval $[0,\infty).$ Thus, we get
\begin{align*}
\mathcal{M}\big\lbrace  {}^{\Psi}\Phi_{p}^{(\alpha,\beta)}&(b;c;z):s\big\rbrace
\\
&=\int_{0}^{\infty}p^{s-1}~ {}^{\Psi}\Phi_{p}^{(\alpha,\beta)}(b;c;z)dp
\\
&=\int_{0}^{\infty}p^{s-1}\sum_{n=0}^{\infty}\frac{{}^{\Psi}B_{p}^{(\alpha,\beta)}(b+n,c-b)}{B(b,c-b)}\frac{z^{n}}{n!}dp
\\
&=\frac{1}{B(b,c-b)}\int_{0}^{1}t^{b-1}(1-t)^{c-b-1}\exp(zt)\int_{0}^{\infty}p^{s-1}{}_{0}\Psi_{1}\left(\alpha,\beta;-\frac{p}{t(1-t)}\right) dp dt.
\end{align*} Since substituting $u=\frac{p}{t(1-t)}$ in the above equation,
\begin{align*}
\int_{0}^{\infty}p^{s-1}{}_{0}\Psi_{1}\left( \alpha,\beta;-\frac{p}{t(1-t)}\right)dp=t^{s}(1-t)^{s}{}^{\Psi}\Gamma^{(\alpha,\beta)}(s).
\end{align*} 
Thus, we get
\begin{align*}
\mathcal{M}\left\lbrace  {}^{\Psi}\Phi_{p}^{(\alpha,\beta)}(b;c;z):s\right\rbrace =\frac{{}^{\Psi}\Gamma^{(\alpha,\beta)}(s)B(b+s,c+s-b)}{B(b,c-b)}\Phi(b+s;c+2s;z).~~~~~~~~~~~~~~~~~~~~~~~~~~~~~~~~~\qedhere
\end{align*}
\end{proof}

\begin{snc}
By the Mellin inversion formula, we have the following complex integral representation for ${}^{\Psi}\Phi_{p}^{(\alpha,\beta)}(b;c;z)$:
\begin{align*}
{}^{\Psi}\Phi_{p}^{(\alpha,\beta)}(b;c;z)=\frac{1}{2\pi i}\int_{-i\infty}^{+i\infty}\frac{{}^{\Psi}\Gamma^{(\alpha,\beta)}(s)B(b+s,c+s-b)}{B(b,c-b)}\Phi(b+s;c+2s;z)p^{-s}ds.
\end{align*}
\end{snc}

\subsection{Transformation formulas}
\begin{teo}
For the $\Psi$-Gauss hypergeometric function, we have the following transformation formula:
	\begin{align*}
	{}^{\Psi}F_{p}^{(\alpha,\beta)}(a,b;c;z)=(1-z)^{-a}\left[ {}^{\Psi}F_{p}^{(\alpha,\beta)}\left( a,c-b;c;\frac{z}{z-1}\right) \right],
	\end{align*}
	where $|\arg(1-z)|<\pi.$
\end{teo} 
\begin{proof}
By writing
\begin{align*}
\left[ 1-z(1-t)\right]^{-a}=(1-z)^{-a}\left( 1+\frac{zt}{1-z}\right)^{-a}
\end{align*} and replacing $t\to 1-t$ in \eqref{j1}, we obtain	
\begin{align*}
{}^{\Psi}F_{p}^{(\alpha,\beta)}(a,b;c;z)=\frac{(1-z)^{-a}}{B(b,c-b)}\int_{0}^{1}t^{c-b-1}(1-t)^{b-1}\left( 1-\frac{zt}{z-1}\right)^{-a}{}_{0}\Psi_{1}\left( \alpha,\beta;-\frac{p}{t(1-t)}\right) dt
\end{align*}
$$\left(Re(p)>0;p=0~\textnormal{and}~|z|<\pi;Re(c)>Re(b)>0\right).$$
Hence,	
\begin{align*}
{}^{\Psi}F_{p}^{(\alpha,\beta)}(a,b;c;z)=(1-z)^{-a}\left[ {}^{\Psi}F_{p}^{(\alpha,\beta)}\left( a,c-b;c;\frac{z}{z-1}\right) \right],
\end{align*}
which completes the proof.
\end{proof}

\begin{teo}
For the $\Psi$-confluent hypergeometric function, we have the following transformation formula:
\begin{align*}
	{}^{\Psi}\Phi_{p}^{(\alpha,\beta)}(b;c;z)=\exp(z)\left[ {}^{\Psi}\Phi_{p}^{(\alpha,\beta)}(c-b;c;-z)\right].
\end{align*}
\end{teo} 
\begin{proof}
	Direct calculations yield	
	\begin{align}\label{i8}
	{}^{\Psi}\Phi_{p}^{(\alpha,\beta)}(b;c;z)&=\sum_{n=0}^{\infty}\frac{{}^{\Psi}B_{p}^{(\alpha,\beta)}(b+n,c-b)}{B(b,c-b)}\frac{z^{n}}{n!}\nonumber\\
	&=\frac{1}{B(b,c-b)}\int_{0}^{1}t^{b-1}(1-t)^{c-b-1}\exp(zt){}_{0}\Psi_{1}\left( \alpha,\beta;-\frac{p}{t(1-t)}\right)dt.
	\end{align}	
Replacing $t\to 1-t$ in \eqref{i8}, we obtain
	\begin{align*}
	{}^{\Psi}\Phi_{p}^{(\alpha,\beta)}(b;c;z)=\exp(z)\left[ {}^{\Psi}\Phi_{p}^{(\alpha,\beta)}(c-b;c;-z)\right],
	\end{align*}
	which completes the proof.
\end{proof}

\end{document}